\documentclass[a4paper,11pt]{article}
\usepackage{graphicx}
\usepackage{amsfonts}
\usepackage{amsmath}
\usepackage{hyperref}
\title{Constructing orientable sequences}
\author{Chris J. Mitchell and Peter R. Wild\\Information Security Group, Royal Holloway,
University of London\\
\href{mailto:me@chrismitchell.net}{me@chrismitchell.net};~~~~\href{mailto:peterrwild@gmail.com}{peterrwild@gmail.com}}
\date{7th January 2022}
\newtheorem{lemma}{Lemma}[section]
\newtheorem{theorem}[lemma]{Theorem}
\newtheorem{corollary}[lemma]{Corollary}
\newtheorem{definition}[lemma]{Definition}
\newtheorem{remark}[lemma]{Remark}
\newtheorem{example}[lemma]{Example}

\newenvironment{proof}[1][Proof]{\begin{trivlist}
\item[\hskip \labelsep {\bfseries #1}]}{\end{trivlist}}
\newcommand{\qed}{\nobreak \ifvmode \relax \else
      \ifdim\lastskip<1.5em \hskip-\lastskip
      \hskip1.5em plus0em minus0.5em \fi \nobreak
      \vrule height0.75em width0.5em depth0.25em\fi}
\begin{document}

\maketitle

\begin{abstract}
This paper describes new, simple, recursive methods of construction for \emph{orientable
sequences}, i.e.\ periodic binary sequences in which any $n$-tuple occurs at most once in a period
in either direction. As has been previously described, such sequences have potential applications
in automatic position-location systems, where the sequence is encoded onto a surface and a reader
needs only examine $n$ consecutive encoded bits to determine its location and orientation on the
surface. The only previously described method of construction (due to Dai et al.)\ is somewhat
complex, whereas the new techniques are simple to both describe and implement.  The methods of
construction cover both the standard `infinite periodic' case, and also the aperiodic, finite
sequence, case. Both the new methods build on the Lempel homomorphism, first introduced as a means
of recursively generating de Bruijn sequences.
\end{abstract}

\section{Introduction} \label{Intro}  \label{sec_intro}

In this paper we are concerned with binary sequences with the property that any $n$-tuple of
consecutive bits occurs either just once in a period, in the case of a periodic infinite sequence,
or just once in a finite sequence (the aperiodic case).  One important special case of such
sequences are the de Bruijn sequences --- see, for example, \cite{Fredricksen82}. These sequences,
sometimes referred to as shift register sequences (see Golomb, \cite{Golomb67}), have been very
widely studied and have a range of applications in coding and cryptography.  One application which
is of particular relevance to this paper, is that of \emph{position location}, i.e.\ use of an
encoding of an $n$-window sequence onto a surface that allows the location of any point on the
surface by examining just $n$ consecutive entries of the sequence (see, for example, Burns and
Mitchell \cite{Burns92,Burns93} and Petriu \cite{C35}).

We are particularly interested in the special case of \emph{orientable sequences} i.e.\ where, for
given order $n$, any $n$-tuple of consecutive values occurs just once in a period \emph{in either
direction}, in the case of an infinite sequence, or just once in either direction in a finite
sequence (the aperiodic case).  These sequences have position-location applications in the case
where the reader of such a sequence wishes to determine both its position and its direction of
travel.  Such sequences were introduced some 30 years ago --- see Burns and Mitchell \cite{Burns93}
and Dai et al.\ \cite{Dai93}. More recent work on the use of sequences for position location
includes that of Szentandr{\'{a}}si et al.\ \cite{Szentandrasi12}, Bruckstein et al.\
\cite{Bruckstein12}, Berkowitz and Kopparty \cite{Berkowitz16}, and Chee et al.
\cite{Chee19,Chee20}.  However, none of this more recent work provides any methods for constructing
orientable sequences.  Observe that orientable sequences are a particular type of universal cycle
--- see Chung et al.\ \cite{Chung92} and Jackson et al.\ \cite{Jackson09}.

Observe that no periodic orientable sequence exists for $n<5$.  A simple example of a periodic
orientable sequence of order $5$ is provided by the sequence of period $m=6$ with generating cycle
$[001101]$, which has optimally long period (see Table~\ref{Dai-table}). Sawada \footnote{See
\url{http://debruijnsequence.org/db/orientable}} provides examples of orders $6$ and $7$ of periods
16 and 36 respectively,[00101011100000] and [0011011110011010100111010010], which were shown to be
optimally long by exhaustive search.

As previously mentioned, Dai et al.\ \cite{Dai93} give a method of construction for orientable
sequences for every $n\geq5$; they also provide an upper bound on the period of such sequences.
Whilst the method of construction is shown to generate sequences of asymptotically optimal periods,
i.e.\ the ratio of the period of a generated sequence with the upper bound tends to 1 as
$n\rightarrow\infty$, the method itself is somewhat complex. For a given order $n$, it involves
working with the set of cycles of length $n$ that are orientable.  These cycles can be divided into
pairs made up of a cycle and its reverse. Using a graph-theoretic argument that is existential
rather than constructive, Dai et al.\ show how one of every pair of the cycles can be joined to
give an orientable sequence of order $n$.

A key motivation for this paper is to work towards addressing the following problem posed by Dai et
al.\ \cite{Dai93}.
\begin{quote}
It is an open problem as to whether a more practical procedure exists for the construction of
orientable sequences that have this asymptotically optimal period.
\end{quote}
It seems likely that the reference to `practical' here means a direct method of construction rather
than one based on an existence proof. We present below a recursive construction method that is much
simpler, and can also generate sequences that are within a fixed factor of asymptotically optimal
period. The fixed factor depends on the period of the `starter sequence', but (using an example
given at \url{http://debruijnsequence.org/db/orientable}) sequences can be constructed with periods
at least 63\% of the maximum possible.  Of course the sequences have period shorter than those due
to Dai et al. \cite{Dai93}, which as noted above are of asymptotically optimal period.

We also examine here the aperiodic case, i.e.\ where the sequence is of finite length.  Whilst this
case was briefly examined by Burns and Mitchell \cite{Burns93}, the only previously known method of
construction was a trivial derivation from the periodic case (see \cite{Burns93} Lemma 1, and also
as outlined in Section~\ref{sec-orientable-construction} below).  Analogously to the periodic case,
we give below a simple, recursive method of constructing such sequences of close to asymptotically
optimal length; 00010111 and 00001101001111 are examples of (optimally long) aperiodic orientable
sequences of orders 4 and 5, respectively.

The remainder of this paper is structured as follows.  In Section~\ref{sec_terminology}, the
terminology used throughout the paper is introduced, together with a range of fundamental results.
The Lempel homomorphism and its application are reviewed in Section~\ref{sec_Lempel}.  A method for
constructing orientable sequences using the Lempel homomorphism is described in
Section~\ref{sec-orientable-construction}; this is followed in
Section~\ref{sec-aperiodic-orientable-construction} by an approach to the construction of aperiodic
orientable sequences.  The paper concludes in Section~\ref{sec-conclusions}.

\section{Terminology and fundamental results}  \label{sec_terminology}

\subsection{Periodic binary sequences with a tuple property}

We are concerned here with periodic binary sequences $S=(s_i)$, where $s_i\in\mathbb{B}=\{0,1\}$,
and where the period $m$ of such a sequence is the smallest positive $m$ such that $s_{i+m}=s_i$
for all $i$. We are particularly interested in finite sub-strings of such sequences ($n$-tuples),
and for $n>0$ we write
\[ \mathbf{s}_n(i) = (s_i,s_{i+1},\ldots,s_{i+n-1}) \]
for the $n$-tuple \emph{appearing at position $i$} in $S$.  We are also interested in the
\emph{weight} of (one period of) a periodic sequence $S=(s_i)$, and we write:
\[ w(S) = \sum_{i=0}^{m-1} s_i \]
where $m$ is the period of $S$.

To simplify certain discussions below, we also introduce the notion of a \emph{generating cycle}.
If $S=(s_i)$ is a periodic binary sequence of period $m$, then the sequence of $m$ values
$s_0,s_1,\ldots,s_{m-1}$ forms the generating cycle of $S$, and clearly the generating cycle
defines the entire sequence. Following Lempel \cite{Lempel70}, we write
$S=[s_0,s_1,\ldots,s_{m-1}]$.

We define an \emph{$n$-window sequence} $S=(s_i)$ (see, for example, \cite{Mitchell96}) to be a periodic
binary sequence of period $m$ with the property that no $n$-tuple appears more than once in a
period of the sequence, i.e.\ with the property that if $\mathbf{s}_n(i)=\mathbf{s}_n(j)$ for some
$i$, $j$, then $i\equiv j\pmod m$.

A \emph{de Bruijn sequence of order $n$} \cite{C45} is then simply an $n$-window sequence of period
$2^n$ (i.e.\ of maximal period), and has the property that every possible $n$-tuple appears once in
a period.

Since we are interested in tuples occurring either forwards or backwards in a sequence we also
introduce the notion of a \emph{reversed} tuple, so that if $\mathbf{u}=(u_0,u_1,\ldots,u_{n-1})$
is a binary $n$-tuple, i.e.\ if $\mathbf{u}\in\mathbb{B}^n$, then
$\mathbf{u}^R=(u_{n-1},u_{n-2},\ldots,u_0)$ is its \emph{reverse}. If a tuple $\mathbf{u}$
satisfies $\mathbf{u}=\mathbf{u}^R$ then we say it is \emph{symmetric}.

The \emph{complement} (or what Lempel \cite{Lempel70} refers to as the \emph{dual}) of a tuple
involves switching every 0 to a 1 and vice versa, and if
$\mathbf{u}=(u_0,u_1,\ldots,u_{n-1})\in\mathbb{B}^n$ we write $\mathbf{\bar{u}}=(u_0\oplus
1,u_1\oplus 1,\ldots,u_{n-1}\oplus 1)$, where here, as throughout, $\oplus$ denotes exclusive-or
(or, equivalently, modulo 2 addition).  In a similar way, we refer to sequences being complementary
if one can be obtained from the other by switching every $1$ to a $0$ and vice versa.

Following Lempel \cite{Lempel70}, we define the \emph{conjugate} of an $n$-tuple to be the tuple
obtained by switching the first bit, i.e.\ if $\mathbf{u}=(u_0,u_1,\ldots,u_{n-1})\in\mathbb{B}^n$,
then the conjugate $\hat{\mathbf{u}}$ of $\mathbf{u}$ is the $n$-tuple
$(u_0\oplus1,u_1,\ldots,u_{n-1})$.

Two $n$-window sequences $S=(s_i)$ and $T=(t_i)$ are said to be \emph{disjoint} if they do not
share an $n$-tuple, i.e.\ if $\mathbf{s}_n(i)\not=\mathbf{t}_n(j)$ for every $i$, $j$.  An
$n$-window sequence is said to be \emph{primitive} if it is disjoint from its complement.

We next give a well known result (closely related to Theorem 2 of Lempel \cite{Lempel70}) showing
how two disjoint $n$-window sequences can be `joined' to create a single $n$-window sequence, if
they contain conjugate $n$-tuples; see also Lemma 3 of Sawada et al.\ \cite{Sawada13}.

\begin{theorem}  \label{thm_cycle_joining}
Suppose $S=(s_i)$ and $T=(t_i)$ are disjoint $n$-window sequences of orders $\ell$ and $m$
respectively.  Moreover suppose $S$ and $T$ contain the conjugate $n$-tuples $\mathbf{u}$ and
$\mathbf{v}$ at positions $i$ and $j$, respectively (i.e.\ $\mathbf{u}=\hat{\mathbf{v}}$).  Then
\[ [s_o,s_1,\ldots,s_{i+n-1},t_{j+n},t_{j+n+1},\ldots,t_{m-1}
,t_0,\ldots,t_{j+n-1},s_{i+n},s_{i+n+1},\ldots,s_{\ell-1}] \]
 is a generating cycle for an
$n$-window sequence of period $\ell+m$.
\end{theorem}

We also introduce a graph of fundamental importance to the study of $n$-window sequences.
\begin{definition}
The de Bruijn-Good graph $\mathcal{G}_n$ \cite{C40} is a directed graph with vertex set
$\mathbb{B}^n$, where for $\mathbf{u},\mathbf{v}\in\mathbb{B}^n$ (where
$\mathbf{u}=(u_0,u_1,\ldots,u_{n-1})$ and $\mathbf{v}=(v_0,v_1,\ldots,v_{n-1})$) there is a
directed edge $\mathbf{u}\rightarrow\mathbf{v}$ if and only if $u_{i+1}=v_i$, $0\leq i\leq n-1$.
\end{definition}
It should be clear that every vertex in $\mathcal{G}_n$ has two incoming edges and two outgoing
edges. It should also be clear that an $(n+1)$-window sequence defines a (directed) cycle in
$\mathcal{G}_n$, (where every $(n+1)$-tuple maps to an edge), and, using the same mapping, a de
Bruijn sequence of order $n+1$ is equivalent to an Eulerian cycle in $\mathcal{G}_n$.  This latter
remark immediately establishes the existence of de Bruijn sequences for every $n$ (given every node
has degree 2). It is also straightforward to see that a de Bruijn sequence of order $n$ defines an
Hamiltonian cycle in $\mathcal{G}_n$, this time using the rather more obvious mapping of $n$-tuples
to nodes.

\subsection{Orientable sequences}  \label{subsec-orientable}

The main focus of this paper is on $n$-window sequences with the property that an $n$-tuple cannot
occur twice within a period in either direction.  To this end we give the following definitions,
following Dai et al.\ \cite{Dai93}.

\begin{definition}
An $n$-window sequence $S=(s_i)$ of period $m$ is said to be an \emph{orientable sequence of order
$n$} (an $\mathcal{OS}(n)$) if, for any $i,j$, $\mathbf{s}_n(i)\not=\mathbf{s}_n(j)^R$.
\end{definition}

We also need the following related concept.

\begin{definition}
A pair of disjoint orientable sequences of order $n$, $S=(s_i)$ and $S'=(s'_i)$, are said to be
\emph{orientable-disjoint} (or simply \emph{o-disjoint}) if, for any $i,j$,
$\mathbf{s}_n(i)\not=\mathbf{s'}_n(j)^R$.
\end{definition}

As noted in Section~\ref{sec_intro}, Dai et al. \cite{Dai93} give an upper bound on the period of
orientable sequences.

\begin{theorem}[Dai et al. \cite{Dai93}]  \label{thm_Dai-bound}
Suppose $S$ is an $\mathcal{OS}(n)$ ($ n\geq 5$).  Then the period of $S$ is at most:
\begin{eqnarray*}
2^{n-1} - 41/9 \times 2^{n/2-1}+n/3+16/9 & \mbox{if} & n\equiv 0 \pmod 4 \\
2^{n-1} - 31/9 \times 2^{(n-1)/2}+n/3+19/9 & \mbox{if} & n\equiv 1 \pmod 4 \\
2^{n-1} - 41/9 \times 2^{n/2-1}+n/6+20/9 & \mbox{if} & n\equiv 2 \pmod 4 \\
2^{n-1} - 31/9 \times 2^{(n-1)/2}+n/6+43/18 & \mbox{if} & n\equiv 3 \pmod 4
\end{eqnarray*}
\end{theorem}

The values arising from Theorem~\ref{thm_Dai-bound} for $5\leq n\leq 9$ are given in
Table~\ref{Dai-table}.

\begin{table}[htb]
\caption{Bounding the period of an $\mathcal{OS}(n)$ (from Theorem~\ref{thm_Dai-bound})}
\label{Dai-table}
\begin{center}
\begin{tabular}{cc} \hline
Order ($n$) & Maximum period $m$ for an $\mathcal{OS}(n)$ \\ \hline
                  5     &      6 \\
                  6     &      17 \\
                  7     &      40 \\
                  8     &      96 \\
                  9     &      206 \\ \hline
\end{tabular}
\end{center}
\end{table}

Observe that the bound of Theorem~\ref{thm_Dai-bound} does not appear to be sharp for $n>5$; as
noted in the introduction, Sawada (see \url{http://debruijnsequence.org/db/orientable}) has shown
that the maximum periods of an OS(6) and OS(7) are 16 and 36 respectively.

\section{The Lempel homomorphism}\label{sec_Lempel}

\subsection{The homomorphism}

The construction method we introduce later in this paper is an application of the homomorphism $D$,
due to Lempel, \cite{Lempel70}.

\begin{definition}[Lempel \cite{Lempel70}]
The mapping $D:\mathbb{B}^n\rightarrow\mathbb{B}^{n-1}$ is as follows. If
$\mathbf{u}=(u_0,u_1,\ldots,u_{n-1})\in\mathbb{B}^n$ then
\[ D(\mathbf{u}) = (u_0\oplus u_1,u_1\oplus u_2,\ldots,u_{n-2}\oplus u_{n-1}) \in\mathbb{B}^{n-1}. \]
\end{definition}

The Lempel homomorphism has the following properties \cite{Lempel70}.
\begin{itemize}
\item $D$ is onto, i.e. $D(\mathbb{B}^n)=\mathbb{B}^{n-1}$ (\cite{Lempel70}, Lemma 1).
\item If $\mathbf{u},\mathbf{v}\in\mathbb{B}^n$ then $D(\mathbf{u})=D(\mathbf{v})$ if and only
    if $\mathbf{u}=\mathbf{v}$ or $\mathbf{u}=\bar{\mathbf{v}}$ (\cite{Lempel70}, Lemma 2).
\item $D$ is a graph homomorphism of $\mathcal{G}_n$ onto $\mathcal{G}_{n-1}$ (\cite{Lempel70},
    Theorem 4).
\end{itemize}


We extend the notation to allow $D$ to be applied to periodic binary sequences in the natural way.
That is, $D$ is a map from the set of periodic binary sequences to itself; the image of a sequence
of period $m$ will clearly have period dividing $m$.  (See also Etzion \cite{C98} for a discussion
of other properties of $D$).  In the natural way we can define $D^{-1}$ to be the `inverse' of $D$,
i.e.\ if $S$ is a periodic binary sequence than $D^{-1}(S)$ is the set of all binary sequences $T$
with the property that $D(T)=S$.

\subsection{Constructing de Bruijn sequences}

We next observe how the Lempel homomorphism can be used to construct a de Bruijn sequence of order
$n+1$ from a de Bruijn sequence of order $n$. Although these results are well-known, we briefly
give them here using our terminology, since they are of key importance for the orientable sequence
construction given below. We first need the following result.

\begin{theorem}\label{thm_window_homo}
Suppose $S=(s_i)$ is an $n$-window sequence of period $m$.  Then:
\begin{itemize}
\item if $w(S)$ is even then $D^{-1}(S)$ consists of a disjoint pair of complementary primitive
    $(n+1)$-window sequences of period $m$, and
\item if $w(S)$ is odd then $D^{-1}(S)$ consists of two different shifts of a single
    $(n+1)$-window sequence of period $2m$ and weight $m$.
\end{itemize}
\end{theorem}

\begin{proof}
Suppose $T=(t_i)\in D^{-1}(S)$. It follows from the definition of $D$ that
\[ t_i=t_0\oplus \bigoplus_{j=0}^{i-1}s_j. \]
We consider the two cases separately.
\begin{itemize}
\item Suppose $w(S)$ is even. Then it follows immediately that $t_{m+i}=t_i$ for all $i$, i.e.\
    $D^{-1}(S)$ consists of a pair of sequences of period $m$. It also immediately follows that
    the two sequences are complementary. Next suppose that
    $\mathbf{t}_{n+1}(i)=\mathbf{t}_{n+1}(j)$ for some $i$,$j$.  Hence
    $D(\mathbf{t}_{n+1}(i))=D(\mathbf{t}_{n+1}(j))$, i.e., by definition of $D$ we know that
    $\mathbf{s}_{n}(i)=\mathbf{s}_{n}(j)$.  Since $S$ is an $n$-window sequence of period $m$,
    it follows that $i\equiv j\pmod m$, and hence $T$ is an ($n+1$)-window sequence.

    To establish primitivity, suppose the opposite, i.e.\ suppose
    $\mathbf{t}_{n+1}(i)=\bar{\mathbf{t}}_{n+1}(j)$ for some $i$,$j$. Then
    $D(\mathbf{t}_{n+1}(i))=D(\bar{\mathbf{t}}_{n+1}(j))$, i.e.\
    $\mathbf{s}_{n}(i)=\mathbf{s}_{n}(j)$.  As previously this implies $i\equiv j\pmod m$, but
    since we know $T$ has period $m$ this immediately gives a contradiction since we assumed
    $\mathbf{t}_{n+1}(i)\not=\mathbf{t}_{n+1}(j)$.

\item Now suppose $w(S)$ is odd. Then it follows immediately that $t_{m+i}=t_i\oplus 1$ for all
    $i$, and hence $t_{2m+i}=t_i$ for all $i$, i.e.\ $T$ has period $2m$.  The fact that
    $D^{-1}(S)$ contains two possible shifts of the same sequence follows by considering that
    $t_0$ can be either 0 or 1.  The fact that $T$ is an $(n+1)$-window sequence follows by
    precisely the same argument as for the even weight case. Finally, it has weight precisely
    half the period since if $\mathbf{u}$ is an $(n+1)$-tuple occurring in $T$, then
    $\bar{\mathbf{u}}$ also occurs in $T$. \qed

\end{itemize}
\end{proof}

We next give two simple examples of the operation of $D^{-1}$.

\begin{example}
First suppose $S=[101]$ (of even weight); then $D^{-1}(S)=\{ [011], [100] \}$.  Alternatively
suppose $S=[100]$ (of odd weight); then $D^{-1}(S)=\{ [100011] \}$.
\end{example}

Since the weight of a binary de Bruijn sequence of order $n>1$ is always even, the above theorem
immediately gives as a corollary the following result due to Lempel \cite{Lempel70}.

\begin{corollary}  \label{thm_Lempel_homo}
Suppose $S=(s_i)$ is a de Bruijn sequence of order $n>1$.  Then $D^{-1}(S)$ consists of a disjoint
    pair of complementary disjoint $(n+1)$-window sequences of period $2^n$.
\end{corollary}

To complete the construction we need the following simple lemma (in essence given in a discussion
in \S IV.A of Lempel \cite{Lempel70}).

\begin{lemma}  \label{lemma_conjugate_pair}
Suppose $S=(s_i)$ is a de Bruijn sequence of order $n>1$, and let $D^{-1}(S)=\{T,T'\}$, where (from
Corollary~\ref{thm_Lempel_homo}) $T$ and $T'$ are a disjoint pair of complementary disjoint
$(n+1)$-window sequences of period $2^n$. Then $T$ and $T'$ contain conjugate $(n+1)$-tuples.
\end{lemma}

\begin{proof}
Consider the two $(n+1)$-tuples consisting of alternating bits, i.e.
$\mathbf{u}=(1010\ldots)$ and $\mathbf{v}=(0101\ldots)$.  They are clearly complementary and so one
occurs in $T$ and the other in $T'$ --- they must both occur in one or other of $T$ and $T'$ since
by an obvious numerical argument $T$ and $T'$ between them contain all $(n+1)$-tuples.  Suppose,
without loss of generality, $\mathbf{u}$ occurs at position $i$ in $T=(t_i)$; then $t_{i-1}=1$
since if $t_{i-1}=1$ then $\mathbf{v}$ occurs at position $i-1$ in $T$, contradicting the
disjointness of $T$ and $T'$.  Hence the conjugate to $\mathbf{v}$ occurs at position $i-1$ in $T$,
giving the desired result. \qed
\end{proof}

It follows immediately that, combining Corollary~\ref{thm_Lempel_homo} with
Theorem~\ref{thm_cycle_joining} and Lemma~\ref{lemma_conjugate_pair}, it is simple to construct a
de Bruijn sequence of order $n+1$ from one of order $n$ by applying the inverse Lempel homomorphism
and then `joining' the two resulting sequences.

\section{Constructing orientable sequences}  \label{sec-orientable-construction}

\subsection{Applying the Lempel homomorphism}

We next show that a similar approach to that described above can be used to construct orientable
sequences of order $n+1$ from one of order $n$.

\begin{theorem}  \label{thm Lempel-orientable}
Suppose $S=(s_i)$ is an orientable sequence of order $n$ and period $m$.  Then:
\begin{itemize}
\item if $w(S)$ is even then $D^{-1}(S)$ consists of an o-disjoint pair of primitive
    complementary orientable sequences of order $n+1$ and period $m$, and
\item if $w(S)$ is odd then $D^{-1}(S)$ consists of two different shifts of a single orientable
    sequence of order $n+1$, period $2m$ and weight $m$.
\end{itemize}
\end{theorem}

\begin{proof}
As previously we consider two cases.
\begin{itemize}
\item If $w(S)$ is even then Theorem~\ref{thm_window_homo} shows that $D^{-1}(S)$ consists of
    an disjoint pair of primitive complementary $(n+1)$-window sequences of order $n+1$ and
    period $m$.  It therefore remains to show that the two sequences are themselves orientable,
    and also that they are o-disjoint.

    First suppose that $T=(t_i)\in D^{-1}(S)$ is not orientable. i.e.
    $\mathbf{t}_{n+1}(i)=\mathbf{t}_{n+1}(j)^R$ for some $i$, $j$.  Then
    $D(\mathbf{t}_{n+1}(i))=D(\mathbf{t}_{n+1}(j)^R)$, i.e.\
    $\mathbf{s}_n(i)=\mathbf{s}_n(j)^R$, contradicting the assumption that $S$ is orientable.

    Next suppose that $T,T'\in D^{-1}(S)$ are not o-disjoint (where $T=(t_i)$ and $T'=(t'_i)$).
We know they are disjoint (from Theorem~\ref{thm_window_homo}), and hence it must hold that
$\mathbf{t}_{n+1}(i)=\mathbf{t}'_{n+1}(j)^R$.  Then
    $D(\mathbf{t}_{n+1}(i))=D(\mathbf{t}'_{n+1}(j)^R)$, i.e.\
    $\mathbf{s}_n(i)=\mathbf{s}_n(j)^R$, again contradicting the assumption that $S$ is
    orientable.

\item If $w(S)$ is odd then Theorem~\ref{thm_window_homo} again almost establishes the result.
    It remains to show that the $(n+1)$-window sequence of period $2m$ is orientable. However,
    this follows by precisely the same argument as used in the previous case. \qed
\end{itemize}
\end{proof}

Clearly this result is not enough on its own to enable construction of `long' orientable sequences,
since, even if $S$ has odd weight and period $m$, then $T\in D^{-1}(S)$ will have weight $m$, i.e.\
it will have odd weight if and only if $m$ is odd.  Moreover, even if $S$ has odd weight and odd
period, then $T\in D^{-1}(S)$ will have period $2m$, and hence $U\in D^{-1}(T)$ will have even
weight.  This is shown by the simple case in Example~\ref{example-n5to6}.

\begin{example}  \label{example-n5to6}
Let $S$ be the $\mathcal{OS}(5)$ of period $m=6$ with generating cycle $[001101]$, mentioned in
Section~\ref{subsec-orientable}.  $S$ has weight 3 (odd) and hence Theorem~\ref{thm
Lempel-orientable} tells us that $D^{-1}(S)$ contains a single $\mathcal{OS}(6)$ of period 12,
namely: $[000100111011]$. However, this has weight 6 (even) so that another application of $D^{-1}$
will yield a complementary pair of $\mathcal{OS}(7)$s of period 12, namely $[000011101001]$ and
$[111100010110]$.
\end{example}

To achieve `period doubling' for multiple iterations of the above construction, two `obvious'
possibilities present themselves:

\begin{itemize}
\item find a sequence $S$ with the property that when applying the construction method
    iteratively, the output complementary pair of sequences contains a conjugate pair of tuples
    (thereby enabling the two sequences to be joined to create a single `double length'
    sequence --- see Theorem~\ref{thm_cycle_joining});
\item find a sequence $S$ (with odd weight) with the property that it is possible to modify the
    double length output sequence $T\in D^{-1}(S)$ to ensure that it too has odd weight.
\end{itemize}
In the next section we exhibit an approach of the second type.

\subsection{An approach to maintaining odd weight}

We first define a method of `extending' orientable sequences of a special type.

\begin{definition}
Suppose $S=(s_i)$ is an orientable sequence of order $n$ and period $m$ with the property that
there is there is exactly one occurrence of $1^{n-4}$ in a period (and hence it contains no longer
runs of $1$s); suppose the generating cycle of $S$ is $[s_0,s_1,\ldots,s_{m-1}]$ where
$s_r=s_{r+1}=\cdots=s_{r+n-5}=1$ for some $r$. Define the function $\mathcal{E}$ (with domain and
range the set of periodic binary sequences) as follows. If $S$ has odd weight then set
$\mathcal{E}(S)=S$, and if $S$ has even weight then define $\mathcal{E}(S)$ to be the sequence with
generating cycle
 \[ [s_0,s_1,s_{r-1},1,s_r,s_{r+1},...s_{m-1}] \]
i.e.\ where the single occurrence of $1^{n-4}$ is replaced with $1^{n-3}$.
\end{definition}

\begin{remark}
Note that, as discussed in Dai et al.\ \cite{Dai93}, it is simple to see that any orientable
sequence can contain at most one occurrence of $1^{n-3}$ in a period.
\end{remark}

We can now state a key result.

\begin{lemma}  \label{lemma-E}
Suppose $S=(s_i)$ is an orientable sequence of order $n$ and period $m$ with the property that
there is exactly one occurrence of $1^{n-4}$ in a period. Then $\mathcal{E}(S)$ is an orientable
sequence of order $n$, period $m$ or $m+1$ (depending on whether $w(S)$ is odd or even) and odd
weight.
\end{lemma}

\begin{proof}
The result clearly holds if $S$ has odd weight, and we thus suppose $S$ has even weight.  The fact
that $\mathcal{E}(S)$ has odd weight follows immediately from the definition. Again by definition
the period of $\mathcal{E}(S)$ divides $m+1$, and is precisely $m+1$ because
$[s_0,s_1,s_{r-1},1,s_r,s_{r+1},...s_{m-1}]$ contains exactly one occurrence of $1^{n-3}$.

It remains to show that $\mathcal{E}(S)$ is an $\mathcal{OS}(n)$.  We only need to examine the
$n$-tuples which include the inserted $1$.  Inserting this single $1$ means that the following
three $n$-tuples that occur in $S$ (where the subscripts are computed modulo $m+1$):
\begin{eqnarray*}
u_0 & = & (s_{r-3}, s_{r-2}, s_{r-1}, 1^{n-4}, s_{r+n-4}), \\
u_1 & = & (s_{r-2}, s_{r-1}, 1^{n-4}, s_{r+n-4}, s_{r+n-3}), \\
u_2 & = & (s_{r-1}, 1^{n-4}, s_{r+n-4}, s_{r+n-3}, s_{r+n-2})
\end{eqnarray*}
are replaced in $\mathcal{E}(S)$ by the following four $n$-tuples:
\begin{eqnarray*}
v_0 & = & (s_{r-3}, s_{r-2}, s_{r-1}, 1^{n-3}), \\
v_1 & = & (s_{r-2}, s_{r-1}, 1^{n-3}, s_{r+n+4}), \\
v_2 & = & (s_{r-1}, 1^{n-3}, s_{r+n-4}, s_{r+n-3}), \\
v_3 & = & (1^{n-3}, s_{r+n-4}, s_{r+n-3}, s_{r+n-2}).
\end{eqnarray*}

Now all four of the $v_i$ tuples contain $1^{n-3}$, and hence they are all distinct and cannot
occur in $S$ (or $S^R$). The only remaining task is to show that $v_0\not=v_3^R$ and
$v_1\not=v_2^R$. However, if $v_0=v_3^R$ or $v_1=v_2^R$ then it immediately follows that $u_1$ is
symmetric, which contradicts the assumption that $S$ is orientable. \qed
\end{proof}

This then enables us to give a means to recursively generate `long' orientable sequences.  We first
define a special class of sequence.

\begin{definition}
An $\mathcal{OS}(n)$ with the property that there is exactly one occurrence of $0^{n-4}$ in a
period is said to be \emph{good}.
\end{definition}

\begin{theorem}
Suppose $S=(s_i)$ is a good $\mathcal{OS}(n)$ of odd weight and period $m$. If $T\in D^{-1}(S)$
then $\mathcal{E}(T)$ is a good $\mathcal{OS}(n+1)$ of odd weight, and period either $2m$ (if $m$
is odd) or $2m+1$ (if $m$ is even).
\end{theorem}

\begin{proof}
The fact that $T$ is an $\mathcal{OS}(n+1)$ of period $2m$ follows immediately from
Theorem~\ref{thm Lempel-orientable}; we also know the weight of $T$ is $m$.  Since $0^{n-4}$ occurs
exactly once in $S$, both $0^{n-3}$ and $1^{n-3}$ occur exactly once in $T$, as
$D^{-1}(0^{n-4})=\{0^{n-3},1^{n-3}\}$. This means that $T$ is good and also the conditions of
Lemma~\ref{lemma-E} apply.  This in turn means that $\mathcal{E}(T)$ is an orientable sequence of
order $n+1$ and odd weight.  The fact that $\mathcal{E}(T)$ is good follows from observing that
applying $\mathcal{E}$ cannot affect the number of occurrences of $0^{n-3}$.  Finally,
$\mathcal{E}(T)$ has period either $2m$ (if $m$ is odd) or $2m+1$ (if $m$ is even), since $T$ has
weight $m$. \qed
\end{proof}

This immediately gives the following result.

\begin{corollary}  \label{Cor-length1}
Suppose $S_n$ is a good $\mathcal{OS}(n)$ of period $m_n$. Recursively define the sequences
$S_{i+1}=\mathcal{E}(D^{-1}(S_i))$ for $i\geq n$, and suppose $S_i$ has period $m_i$ ($i>n$). Then,
$S_i$ is a good $\mathcal{OS}(i)$ for every $i$, and for every $j\geq 0$:
\begin{itemize}
\item if $m_n$ is odd, $m_{n+2j+1}=2m_{n+2j}$ and $m_{n+2j+2}=2m_{n+2j+1}+1$;
\item if $m_n$ is even, $m_{n+2j+1}=2m_{n+2j}+1$ and $m_{n+2j+2}=2m_{n+2j+1}$.
\end{itemize}
\end{corollary}

\begin{proof}
If $m_i$ is odd for any $i\geq m$, then $D^{-1}(S_i)$ will have odd weight and hence
$S_{i+1}=D^{-1}(S_i)$; that is, $S_{i+1}$ will have even period ($2m_i$).  By similar reasoning,
$S_{i+2}$ will have odd period ($2m_{i+1}+1=4m_i+1$).  This immediately yields the result. \qed
\end{proof}

Simple numerical calculations give the following.

\begin{corollary}
Suppose the sequences $(S_i)$ are defined as in Corollary~\ref{Cor-length1}.  Then
\begin{itemize}
\item if $m_n$ is odd, $m_{n+2j}=2^{2j}m_n+(2^{2j}-1)/3$ and
    $m_{n+2j+1}=2^{2j+1}m_n+(2^{2j+1}-2)/3$;
\item if $m_n$ is even, $m_{n+2j}=2^{2j}m_n+(2^{2j+1}-2)/3$ and
    $m_{n+2j+1}=2^{2j+1}m_n+(2^{2j+2}-1)/3$.
\end{itemize}
\end{corollary}

We conclude by giving a simple example of how the above process can be used to generate an infinite
family of orientable sequences.

\begin{example}  \label{periodic-example}
$[001010111]$ is the generating cycle of an $\mathcal{OS}(6)$ of period 9, which is good since it
contains exactly one instance of $0^2$.  It also has odd weight.  So it can be used as $S_6$ for
the first application of $D^{-1}$.  This results in a good $\mathcal{OS}(7)$ of period 18 with
generating cycle: $[000110010 111001101]$.  This has weight 9 (which is odd), i.e.\
$\mathcal{E}(D^{-1}(S_6))=D^{-1}(S_6)$, and so $S_7=[000110010 111001101]$.  We next have
\[  D^{-1}(S_7) = [000010001 101000100 111101110 010111011] \]
which has even weight and hence we need to insert an extra 1 after the unique sequence of four $1$s,
i.e.\
\[ S_8 = [000010001 101000100 1111101110 010111011]. \]
Continuing in this way we obtain sequences with the periods listed in Table~\ref{new-table}.
\end{example}
\begin{table}[htb]
\caption{A family of orientable sequences}  \label{new-table}
\begin{center}
\begin{tabular}{ccc} \hline
Order ($n$) & Period ($m_n$) & Remarks \\ \hline
                  6     &      9  & \\
                  7     &      18 & \mbox{No need to insert an extra 1 as the weight}\\
                        &         & \mbox{is odd since [001010111] has odd period} \\
                  8     &      37 & \mbox{After adding the extra 1} \\
                  9     &      74 & \mbox{No need to insert an extra 1} \\
                  10     &     149 & \\
                  $\cdots$  & $\cdots$ \\
                  $6+2j$ & $9(2^{2j})+(2^{2j}-1)/3$ \\
                  $6+2j+1$ & $9(2^{2j+1})+(2^{2j+1}-2)/3$ \\
  \hline
\end{tabular}
\end{center}
\end{table}

For the sequences in Example~\ref{periodic-example} we thus have $m_n>9\times 2^{n-6}$ and from Dai
et al.\ \cite{Dai93} we know that $m_n<2^{n-1}$ for any orientable sequence of order $n$.  That is,
even for this simple example, the sequences obtained have periods at least 9/32 of the optimal
values. Clearly sequences with periods closer to the optimal values can be obtained if the `starter
sequence' has period larger than the value given in the table\footnote{An example of a good
$\mathcal{OS}(8)$ of period 80 has been found by Sawada --- see
\url{http://debruijnsequence.org/db/orientable} --- yielding sequences of length at least 63\% of
the longest possible such sequence.}, although the generated sequences will never by asymptotically
optimal in length (unlike the sequences of Dai et al.\ \cite{Dai93}). The sequence of period 9 was
found by hand, and in the absence of a systematic search it is not clear whether a good
$\mathcal{OS}(6)$ of period greater than 9 exists.

\section{The aperiodic case}  \label{sec-aperiodic-orientable-construction}

\subsection{Introduction and definitions}  \label{subsec-aperiodic-intro}

Up to this point we have only considered \emph{periodic} sequences, i.e.\ infinite binary sequences
which repeat after a finite period.  However, many of the ideas we have thus far discussed also
apply to the \emph{aperiodic} case, i.e. where we are dealing with a single finite sequence.  If
$S=(s_0,s_1,\ldots,s_{\ell-1})$ is a binary sequence of length $\ell$, i.e.\ $s_i\in\mathbb{B}$ for
$0\leq i<\ell$, then $S$ is an aperiodic orientable sequence of order $n$ (an $\mathcal{AOS}(n)$)
if and only if the collection of $2\ell-2n+2$ $n$-tuples $\mathbf{s}_n(i)$ and $\mathbf{s}_n(i)^R$ ($0\leq
i\leq \ell-n$) are all distinct.

As noted in Section~\ref{sec_intro}, sequences of both periodic and aperiodic type have potential
applications in position-location applications where the sequence is encoded onto a surface which
may be read in either direction, and reading $n$ digits reveals the location of the reader and the
direction of travel. Whether the periodic or aperiodic sequences are more appropriate depends on
whether the surface on which the sequence is encoded forms a closed loop, e.g.\ when the surface is
a cylinder, or not.

This case was briefly considered by Burns and Mitchell \cite{Burns93}, who give some simple results
on the lengths of the longest such sequences, obtained from computer searches --- see
Table~\ref{Burns-table}, where it is claimed that the for $4\leq n\leq 7$ the length given is the
length of the longest such sequence. Note that such sequences are referred to there as binary
aperiodic 2-orientable window sequences.

\begin{table}[htb]
\caption{Existence of aperiodic orientable sequences}  \label{Burns-table}
\begin{center}
\begin{tabular}{cc} \hline
Order ($n$) & Sequence length ($\ell$) \\ \hline
                  4     &      8 \\
                  5     &      14 \\
                  6     &      26 \\
                  7     &      48 \\
                  8     &      108 \\
                  9     &      210 \\
                  10    &      440 \\
                  11    &      872 \\
                  12    &      1860 \\
                  13    &      3710 \\
                  14    &      7400 \\
                  15    &      15467 \\
                  16    &      31766 \\ \hline
\end{tabular}
\end{center}
\end{table}

It follows immediately from the definitions that if $S=(s_i)$ is an $\mathcal{OS}(n)$ of period
$m$, then $(s_0,s_1\ldots,s_{n+m-2})$ is an $\mathcal{AOS}(n)$ of length $m+n-1$ (see also
\cite{Burns93}).

Finally, analogously to the periodic case, we define:
\begin{itemize}
\item a pair of $\mathcal{AOS}(n)$s to be \emph{disjoint} if they do not share an $n$-tuple;
\item a pair of $\mathcal{AOS}(n)$s to be \emph{o-disjoint} if they do not share an $n$-tuple
    in either direction; and
\item an aperiodic sequence to be \emph{primitive} if it is disjoint from its complement.
\end{itemize}

\subsection{Applying the Lempel homomorphism}

The Lempel homomorphism applies equally in the aperiodic case, and the following result analogous
to Theorem~\ref{thm Lempel-orientable} holds.  As the operator $D$ acts in the same way on $(n+1)$-windows in the aperiodic case as the periodic case the proof follows using precisely the same arguments
as in Theorems~\ref{thm_window_homo} and \ref{thm Lempel-orientable}.

\begin{theorem}  \label{thm Lempel-orientable-aperiodic}
Suppose $S=(s_i)$ is an $\mathcal{AOS}(n)$ of length $\ell$.  Then $D^{-1}(S)$ consists of an
o-disjoint pair of primitive complementary $\mathcal{AOS}(n+1)$s of length $\ell+1$.
\end{theorem}

Of course, this result does not enable us to generate long aperiodic orientable sequences.  As in
the periodic case, we need to find a way to combine the pair of sequences output from the inverse
Lempel homomorphism.  Fortunately, as we show below, if you start the iterative process with a
sequence with very special properties then this can be achieved.

\subsection{A special case}

We start by considering a very special type of $\mathcal{AOS}(n)$.  We first need the following.

\begin{definition}
If $S=(s_i)$ is an $\mathcal{AOS}(n)$, $n>1$, of length $\ell$ with the property that the first
$n-1$ bits are $0$s and the last $n-1$ bits are $1$s, i.e. $S=0^{n-1}\cdots 1^{n-1}$, then we say
$S$ is \emph{ideal}.
\end{definition}

We can now state the following simple result, which follows immediately from Theorem~\ref{thm
Lempel-orientable-aperiodic} and the definition of $D$.

\begin{lemma} \label{lemma-iterate-aperiodic}
Suppose $S=(s_i)$ is an ideal $\mathcal{AOS}(n)$.  Then:
\[ D^{-1}(S)=\{ (0^{n}\cdots a_{n}), (1^{n}\cdots \bar{a}_{n}) \} \]
where $a_{n}$ is a subsequence consisting of $n$ alternating bits (whether it starts with a $0$ or
a $1$ is immaterial).
\end{lemma}

This then leads to the following main result.

\begin{theorem}  \label{theorem-iterate-aperiodic}
Suppose $S=(s_i)$ is an ideal $\mathcal{AOS}(n)$ of length $\ell$, and let
$D^{-1}(S)=\{T,\bar{T}\}$ where $T=(0^{n}\cdots a_{n})$ as in Lemma~\ref{lemma-iterate-aperiodic}.  Let $U=\bar{T}^R$.
\begin{itemize}
\item If $n$ is even, let $V$ be the sequence of length $2\ell-n+2$ consisting of the $\ell+1$
    bits of $T$ followed by the final $\ell-n+1$ bits of $U$ (i.e.\ $U$ with the first $n$ bits omitted).
\item If $n$ is odd, let $V$ be the sequence of length $2\ell-n+3$ consisting of the $\ell+1$
    bits of $T$ followed by the final $\ell-n+2$ bits of $U$ (i.e.\ $U$ with the first $n-1$ bits omitted).
\end{itemize}
In both cases $V$ is an ideal $\mathcal{AOS}(n+1)$.
\end{theorem}

\begin{proof}
Without loss of generality we suppose throughout that $a_{n}$ starts with a 0 (and hence ends with a 0/1 if $n$ is odd/even).

First suppose that $n$ is even.  Then, since $\bar{T}=1^{n}\cdots \bar{a}_{n}$, $U = \bar{a}_{n}\cdots
1^{n}$, the first $n$ bits of $U$ match the final $n$ bits of $T$. It follows that the
merging of the two sequences to create $V$ does not introduce any additional `new' $(n+1)$-tuples, i.e.\ the
set of $(n+1)$-tuples in $V$ is equal to those appearing in $T$ and $U$.  Hence, since
$U=\bar{T}^R$, and $T$ and $\bar{T}$ are an o-disjoint pair of primitive complementary
$\mathcal{AOS}(n+1)$s (from Theorem~\ref{thm Lempel-orientable-aperiodic}), it follows that $V$ is
an $\mathcal{AOS}(n+1)$.  The fact that it is ideal follows immediately from its method of
construction, and its length is equal to the sum of the lengths of $T$ and $U$ minus $n$, the number of
`overlapped' bits, i.e. $2(\ell+1)-n=2\ell-n+2$.

Now suppose that $n$ is odd.  In this case $U = \bar{a}_{n}\cdots 1^{n}$, and so the final $n-1$
bits of $U$ match the first $n-1$ bits of $T$. The `merging' of $T$ and $U$ to create $V$
introduces a single additional `new' $(n+1)$-tuple in $V$, namely the $(n+1)$-bit alternating tuple
starting with a 0 --- all other tuples occur in $T$ or $U$; similarly, the only additional `new'
$(n+1)$-tuple in $V^R$ is the $(n+1)$-bit alternating tuple starting with a 1 (since $n+1$ is
even).  Neither of these tuples could appear in $T$ or $U$ (or their reverses) since the image
under $D$ of these $(n+1)$-tuples is $1^n$, which cannot appear in $S$ since it is symmetric. It
thus again follows that $V$ is an ideal $\mathcal{AOS}(n+1)$, in this case of length equal to the
sum of the lengths of $T$ and $U$ minus $n-1$, the number of `overlapped' bits, i.e.
$2(\ell+1)-(n-1)=2\ell-n+3$. \qed
\end{proof}

The above construction clearly gives an iterative method of computing an $\mathcal{AOS}(n)$ for
arbitrary $n$, given an ideal $\mathcal{AOS}$ to act as a `starter' in the construction.  The
length of the sequences obtained is given by the following lemma.

\begin{lemma}  \label{lemma-aperiodic-length}
Suppose $S_n$ is an ideal $\mathcal{AOS}(n)$ of length $\ell_n$, and moreover suppose that
$S_{m+n}$ is an ideal $\mathcal{AOS}(m+n)$ of length $\ell_{m+n}$ obtained from $S_n$ using $m$
iterations of the approach given in Theorem~\ref{theorem-iterate-aperiodic}.  Then
\[\ell_{m+n} = 2^m(\ell_n-n+1) + x_m/3 + m+n-1. \]
where $x_m$ is one of $2^m-1$ ($n$ even, $m$ even), $2^m-2$ ($n$ even, $m$ odd), $2^{m+1}-2$ ($n$
odd, $m$ even), or $2^{m+1}-1$ ($n$ odd, $m$ odd).
\end{lemma}

\begin{proof}
For any sequence $S_{r}$ ($r\geq n$) we consider the value $\nu_{r}$, i.e.\ the number of
$r$-tuples appearing in $S_r$. Clearly $\nu_r=\ell_r-r+1$.  From
Theorem~\ref{theorem-iterate-aperiodic}, we immediately have that $\nu_{n+1}=2\nu_n$ if $n$ is
even, and $\nu_{n+1}=2\nu_n+1$ if $n$ is odd. The result then follows from some simple
calculations. \qed
\end{proof}

We now give a simple example of the use of the above iterative construction, yielding an infinite
family of aperiodic orientable sequences.

\begin{example}
We start by observing that $S_2=01$ is an ideal $\mathcal{AOS}(2)$ of length 2. This is clearly
optimally long. Now $D^{-1}(S)=\{001,110\}$, i.e. $T=001$ and $U=011$.  Since 2 is even, we overlap
the sequences by $n=2$ positions to obtain $S_3=0011$.  Note that $S_3$ is also of optimal length since
there are only four asymmetric 3-tuples, two of which appear.

Repeating the construction, applying $D^{-1}$ to $0011$ gives $T=00010$ and $U=10111$.  Overlapping
them by 2($=n-3$) bit-positions we get $S_4=00010111$ of length 8.  This involves adding the extra
4-tuple $0101$, which could not appear in either $T$ or $U$ as its image under $D$ is $111$, which
is symmetric.  $S_4$ is also optimally long according to Table~\ref{Burns-table}.

We next obtain $S_5=00001101001111$ of length 14. Continuing this process gives sequences of the
lengths in Table~\ref{AOS-family-table}.
\end{example}

\begin{table}[htb]
\caption{A family of aperiodic orientable sequences}  \label{AOS-family-table}
\begin{center}
\begin{tabular}{cc} \hline
Order ($n$) & Sequence length ($\ell$) \\ \hline
                  2     &      2 \\
                  3     &      4 \\
                  4     &      8 \\
                  5     &      14 \\
                  6     &      26 \\
                  7     &      48 \\
                  8     &      92 \\
                  9     &      178 \\
                  10    &      350 \\
                  $\cdots$ & $\cdots$ \\
                  $2r$ & $2^{2r-2}+(2^{2r-2}-1)/3+2r-1$ \\
                  $2r+1$ & $2^{2r-1}+(2^{2r-1}-2)/3+2r$ \\ \hline
\end{tabular}
\end{center}
\end{table}

Inspection of Table~\ref{AOS-family-table} reveals that for $n\leq7$ the sequences are optimally
long (according to Table~\ref{Burns-table}), but for larger values of $n$ they are not. However,
they are very simple to construct, and from Lemma~\ref{lemma-aperiodic-length} (see also
Table~\ref{AOS-family-table}) the length of the sequence $S_n$ is greater than $2^{n}/3$. A simple
upper bound (see Lemma 15 of \cite{Burns93}) means that the length of any $\mathcal{AOS}(n)$ is at
most $2^{n-1}-2^{\lfloor(n-1)/2\rfloor}+n-1$; that is, the sequences generated by this approach
have lengths at least 2/3 of the optimal values.

\section{Conclusions and possible future work}  \label{sec-conclusions}

We have described how the Lempel homomorphism can be applied to recursively generate infinite
families of both periodic and aperiodic orientable sequences.  We have given examples of infinite
families of orientable sequences, both periodic and aperiodic, generated using the construction
methods and having, in both cases, close to optimal period/length --- the periodic sequences have
period at least 63\% of the optimal value, and the aperiodic sequences have length at least 2/3 of
optimal. Moreover, sequences with greater period/length can be obtained should longer `starter'
sequences for the recursions be chosen. The method of construction in both cases is direct very
simple, partially answering the question posed by Dai et al.\ \cite{Dai93} and quoted in
Section~\ref{sec_intro}. They only \emph{partially} answer the question since the periods/lengths
of the sequences produced are not asymptotically optimal.  The methods are also of low complexity
in terms of time --- they are trivially linear in the sequence length; however, the storage
complexity is high since the entire sequence needs to be available to perform the recursion
operation.

Since in both cases the sequences are generated using a simple recursive approach, it may well be
possible to devise simple encoding and decoding methods, i.e.\ algorithms that, for an
$\mathcal{OS}(n)$ or an $\mathcal{AOS}(n)$, enable a position value to be converted into the
$n$-tuple that occurs in that position in the sequence or its reverse (encoding) or vice versa
(decoding).  Such algorithms are clearly of value in potential position-location applications of
orientable sequences --- see, for example, \cite{Burns93,Mitchell96}.

Devising such algorithms is left for future work.  Other possible directions for future research
include generalising the construction methods given in this paper, both to arbitrary size alphabets
and to the multi-dimensional case.

\section*{Acknowledgements}

The authors would like to thank Joe Sawada and the anonymous referees for their valuable
corrections and suggestions for improvement.

\providecommand{\bysame}{\leavevmode\hbox to3em{\hrulefill}\thinspace}
\providecommand{\MR}{\relax\ifhmode\unskip\space\fi MR }
\providecommand{\MRhref}[2]{%
  \href{http://www.ams.org/mathscinet-getitem?mr=#1}{#2}
}
\providecommand{\href}[2]{#2}

\end{document}